\documentclass[sn-mathphys,Numbered]{sn-jnl}% Math and Physical Sciences Reference Style
\usepackage{graphicx}%
\usepackage{multirow}%
\usepackage{amsmath,amssymb,amsfonts}%
\usepackage{amsthm}%
\usepackage{mathrsfs}%
\usepackage[title]{appendix}%
\usepackage{xcolor}%
\usepackage{textcomp}%
\usepackage{manyfoot}%
\usepackage{booktabs}%
\usepackage{algorithm}%
\usepackage{algorithmicx}%
\usepackage{algpseudocode}%
\usepackage{listings}%
\newcommand{\ua}{\underline{a}}

\newcommand{\ut}{\underline{t}}
\newcommand{\us}{\underline{s}}
\newcommand{\ur}{\underline{r}}
\newcommand{\ux}{\underline{x}}
\newcommand{\uX}{\underline{X}}
\newcommand{\uy}{\underline{y}}

\newcommand{\uz}{\underline{z}}

\newcommand{\uu}{\underline{u}}
\newcommand{\uU}{\underline{U}}
\newcommand{\uw}{\underline{w}}
\newcommand{\up}{\underline{p}}
\newcommand{\uW}{\underline{W}}
\newcommand{\uP}{\underline{P}}

\newcommand{\rta}{\rightarrow}
\newcommand{\conv}{\rightsquigarrow}
\newcommand{\Lstar}{\mathcal{L}^{\omega}_*}

\newcommand{\IL}{\mathrm{IL}^{\omega}_*}
\newcommand{\HA}{\mathrm{HA}^{\omega}_*}
\newcommand{\ILp}{\mathrm{IL}}

\newcommand{\AC}{\mathrm{AC}^{\omega}_*}

\newcommand{\Red}{\mathrm{Red}}

\theoremstyle{thmstyleone}%
\newtheorem{theorem}{Theorem}%  meant for continuous numbers
\newtheorem{definition}{Definition}
\newtheorem{remark}{Remark}

\newtheorem{lemma}{Lemma}
\newtheorem{corollary}{Corollary}

\begin{document}

\title[Herbrandized modified realizability]{Herbrandized modified realizability}

%%=============================================================%%
%% Prefix	-> \pfx{Dr}
%% GivenName	-> \fnm{Joergen W.}
%% Particle	-> \spfx{van der} -> surname prefix
%% FamilyName	-> \sur{Ploeg}
%% Suffix	-> \sfx{IV}
%% NatureName	-> \tanm{Poet Laureate} -> Title after name
%% Degrees	-> \dgr{MSc, PhD}
%% \author*[1,2]{\pfx{Dr} \fnm{Joergen W.} \spfx{van der} \sur{Ploeg} \sfx{IV} \tanm{Poet Laureate} 
%%                 \dgr{MSc, PhD}}\email{iauthor@gmail.com}
%%=============================================================%%

\author*[1,2]{\fnm{Gilda} \sur{Ferreira}}\email{gmferreira@fc.ul.pt}

\author[3]{\fnm{Paulo} \sur{Firmino}}\email{fc49883@alunos.ciencias.ulisboa.pt}
%\equalcont{These authors contributed equally to this work.}

%\author[1,2]{\fnm{Third} \sur{Author}}\email{iiiauthor@gmail.com}
%\equalcont{These authors contributed equally to this work.}

\affil*[1]{\orgdiv{DCeT}, \orgname{Universidade Aberta}, \postcode{1269-001}, \state{Lisboa}, \country{Portugal}}

\affil[2]{\orgname{Centro de Matem\'atica, Aplica\c{c}\~oes Fundamentais e Investiga\c{c}\~ao Operacional}, \orgaddress{\street{Faculdade de Ci\^encias, Universidade de Lisboa}, \postcode{1749-016}, \state{Lisboa}, \country{Portugal}}}

\affil[3]{\orgdiv{Departamento de Matem\'atica}, \orgname{Faculdade de Ci\^{e}ncias}, \orgaddress{\street{Universidade de Lisboa}, \postcode{1749-016}, \state{Lisboa}, \country{Portugal}}}

%%==================================%%
%% sample for unstructured abstract %%
%%==================================%%

\abstract{Realizability notions in mathematical logic have a long history, which can be traced back to the work of Stephen Kleene in the 1940s, aimed at exploring the foundations of intuitionistic logic. Kleene's initial realizability laid the ground for more sophisticated notions such as Kreisel's modified realizability and various modern approaches. In this context, our work aligns with the lineage of realizability strategies that emphasize the accumulation, rather than the propagation of precise witnesses. In this paper, we introduce a new notion of realizability, namely \emph{herbrandized modified realizability}. This novel form of (cumulative) realizability, presented within the framework of semi-intuitionistic logic is based on a recently developed \emph{star combinatory calculus}, which enables the gathering of witnesses into nonempty finite sets. We also show that the previous analysis can be extended from logic to (Heyting) arithmetic.}


%%\pacs[JEL Classification]{D8, H51}


\maketitle

\section{Introduction}\label{Introduction}

Notions of realizability have a rich history, with the first generation dating back to the 40s, with the pioneering Kleene's realizability (1945) \cite{Kleene(1945)} and Kreisel's modified realizability (1962) \cite{Kreisel(1962)}. These realizability strategies, closely connected with the Brouwer-Heyting-Kolmogorov foundational interpretation of intuitionistic logic developed in the early 1900s, aim to explicitly reveal the constructive content of proofs, decide disjunctions, and offer precise witnesses for existential statements. Later, a second generation of realizability strategies emerged, shifting the focus from precise witnesses to bounds for those witnesses. Examples of such realizability strategies include the bounded modified realizability (BMR) \cite{FerreiraNunes(06)} and the confined modified realizability (CMR) \cite{FerreiraOliva(10)}, which draw inspiration from earlier works of Ulrich Kohlenbach, Fernando Ferreira, and Paulo Oliva. In 1996, Kohlenbach introduced the monotone functional interpretation \cite{Kohlenbach(96)}, and in 2005, Ferreira and Oliva introduced the bounded functional interpretation \cite{FerreiraOliva(05)}. [Although with different approaches] Both strategies differ from the well-known G\"{o}del's functional (Dialectica) interpretation since they do not rely on precise witnesses but instead on majorizing functionals. BMR and CMR apply the concept of majorizability in all finite types to the realizability method, also delivering bounds instead of precise realizers.

In 2017, a new type of combinatory calculus was introduced, incorporating star types \cite{FerreiraFerreira(17)} that enable the formation of sets of potential witnesses. The focus can now shift from bounds for witnesses to finite sets of possible witnesses. Based on this calculus, a cumulative/herbrandized functional interpretation was developed within the framework of classical first-order predicate logic, drawing on the corresponding Shoenfield version of the Dialectica interpretation. This cumulative interpretation followed in the footsteps of a previous work of Benno van den Berg, Eyvind Briseid and Pavol Safarik \cite{BergBriseidSafarik(12)} in the context of nonstandard arithmetic. For a comparison between this later approach and the one pursued in \cite{FerreiraFerreira(17)}, see \cite{Borges(16)}. Another cumulative interpretation based on (this time denumerable) sets can be found in \cite{AvigadFeferman(98)}.

Based on the above prior research, this paper shows how the star combinatory calculus can serve as the basis for a novel modified realizability approach, namely the herbrandized modified realizability, within the domains of semi-intuitionistic logic and Heyting arithmetic. 

\vspace{0.5cm}

{{\bf{Overview:}} The paper is structured as follows. In the next section we recall the star combinatory calculus and its nice proof-theoretic properties. In Section 3 we introduce the herbrandized modified realizability in the semi-intuitionistic context and establish the corresponding soundness and characterization theorems. We finish the paper with some final notes, including the possibility of extending the previous analysis to the Heyting arithmetic context.
	
\section{Background: star combinatory calculus}\label{star}

Let us fix a language $\mathcal{L}$ of pure first-order logic containing at least one constant symbol. We are going to define the language $\Lstar$, a language in all finite types (based on a given ground type $G$).

\begin{definition}
	The \emph{types} are inductively generated as follows:
	
	\begin{itemize}
		\item[(i)] The \emph{ground type} ($G$) is a type;
		\item[(ii)] If $\sigma$ and $\tau$ are types, $\sigma \rightarrow \tau$ is a type;
		\item[(iii)] If $\sigma$ is a type, $\sigma^*$ is a type (a \emph{star type}).
	\end{itemize}
\end{definition}

%\begin{remark}
Each type represents a class of objects. As usual, informally  the type $\sigma \rightarrow \tau$ represents the functionals from objects of type $\sigma$ to objects of type $\tau$. The intended meaning of $\sigma^*$ is the type of all finite non-empty subsets of objects of type $\sigma$.
%\end{remark}

\begin{remark}\label{generictype}
	Any type is of the form $\sigma_1 \rta (\sigma_2 \rta (\cdots \rta (\sigma_n \rta \rho)\cdots))$ with $\rho$ the type $G$ or a star type. We follow the convention of associating $\rta$ to the right, abbreviating the above type by $\sigma_1 \rta \cdots \rta \sigma_n \rta \rho$.
\end{remark}

%\paragraph{Constants of $\Lstar$}:

For each functional symbol $f$ of $\mathcal{L}$, with arity $n$ ($n\geq 0$), $f$ denotes a constant of $\Lstar$, of type $G \rta \cdots \rta G \rta G$ with $n$ arrows.

Additionally, we have the \emph{logical constants} or \emph{combinators}:
\begin{itemize}
	\item[] $\Pi_{\sigma,\tau}$ of type $\sigma \rightarrow \tau \rightarrow \sigma$, for each pair of types $\sigma$, $\tau$;
	\item[] $\Sigma_{\rho,\sigma,\tau}$ of type $(\rho \rightarrow \sigma \rightarrow \tau) \rightarrow (\rho \rightarrow \sigma) \rightarrow \rho \rightarrow \tau$, for each triple of types $\rho$, $\sigma$, $\tau$;
\end{itemize}

%\vspace{10pt}
and the \emph{star constants}:
\begin{itemize}
	\item[] $\mathfrak{s}_{\sigma}$ of type $\sigma \rightarrow \sigma^*$, for each type $\sigma$;
	\item[] $\cup_{\sigma}$ of type $\sigma^* \rightarrow \sigma^* \rightarrow \sigma^*$, for each type $\sigma$;
	\item[] $\bigcup_{\sigma,\tau}$ of type $\sigma^* \rightarrow (\sigma \rightarrow \tau^*) \rightarrow \tau^*$, for each pair of types $\sigma$, $\tau$.
\end{itemize}

%\begin{remark} 
%\end{remark}

\begin{definition}
	
	The terms of $\Lstar$ are inductively generated as follows:
	
	\begin{itemize}
		\item[(i)] Constants are terms;
		\item[(ii)] For each type $\sigma$, there exists a countable infinite number of variables of type $\sigma$ (usually denoted by $x, y, z \ldots$). Variables are terms;
		\item[(iii)] If $t$ is a term of type $\sigma \rightarrow \tau$ and $q$ is a term of type $\sigma$, then  $tq$ is a term of type $\tau$.
	\end{itemize}
\end{definition}

The fact that the term $t$ has type $\sigma$ is sometimes denoted by $t^{\sigma}$ or $t:\sigma$.

\vspace{.3cm}

Concerning the intended meaning of the star constants above: $\mathfrak{s}_{\sigma}t$ represents a set with a unique element $t$; $\cup_{\sigma}tq$ represents the union of the sets $t$ and $q$ and  $\bigcup_{\sigma,\tau}tq$ for $t: \sigma^*$ and $q: \sigma \rta \tau^*$ represents the union $\bigcup_{w\in t} qw$.

%\paragraph{Conversions:}

\vspace{.3cm}

The star combinatory calculus has the following conversions:
\begin{itemize}
	\item[] $\Sigma_{\rho,\sigma,\tau} tqr \rightsquigarrow tr(qr)$ ($t: \rho \rightarrow \sigma \rightarrow \tau,\,q: \rho \rightarrow \sigma,\,r: \rho$)
	\item[] $\Pi_{\sigma,\tau}tq \rightsquigarrow t$ ($t: \sigma,\,q: \tau$)
	\item[] $\bigcup_{\sigma,\tau}(\mathfrak{s}_\sigma t)q \rightsquigarrow qt$ ($t: \sigma, \,q: \sigma \rta \tau^*$)
	\item[] $\bigcup_{\sigma,\tau}(\cup_\sigma tq)r \rightsquigarrow \cup_\tau(\bigcup_{\sigma,\tau} tr)(\bigcup_{\sigma,\tau} qr)$ ($t: \sigma^*, \,q: \sigma^*,\, r: \sigma \rta \tau^*$)
\end{itemize}

%\begin{obs}
%Cada conversão corresponde a uma igualdade satisfeita pelos operadores representados pelas constantes.
%\end{obs}

%\begin{dfn}
Given two terms $t$ and $q$ of the same type, we say that we have a \emph{one-step reduction}, which we denote by $t \succ_1 q$, if $q$ is obtained from $t$ applying one of the above conversions to a subterm of $t$. We write $t \succeq q$, and say that \emph{$t$ reduces to $q$} (or we have a \emph{reduction} from $t$ to $q$) if $q$ is obtained from $t$ by an arbitrary number (possibly zero) of one-step reductions.
%\end{dfn}

A term is said to be \emph{normal} if it is not possible to apply any one-step reduction on the term.
A term $t$ is said to be \emph{strongly normalizable} if there is no infinite chain of one-step reductions starting from $t$. %Dado um term $t$ fortemente normalizável, denotamos $\nu(t)$ o maior número de reduções de um passo necessárias para obterms a partir de $t$ um term normal.

The star combinatory calculus has the following three key properties, whose proofs can be found on \cite{FerreiraFerreira(17)}.

\begin{theorem}
	The star combinatory calculus has the strong normalization property.
\end{theorem}

\begin{theorem} \label{CR}
	The star combinatory calculus has the Church-Rosser property, that is, each term has a unique normal form.
\end{theorem}

%\begin{teor}
%If $t$ is a normal closed term of type $G$ then $t$ is a first-order term of the language $\mathcal{L}$.
%\end{teor}

\begin{definition}
	A term $t$ of star type is said to be \emph{set-like} if it is constructed from terms of the form $\mathfrak{s}q$ and the constant $\cup$.
\end{definition}

\begin{theorem}
	If $t$ is a normal closed term of star type, then $t$ is set-like. %e $\mathrm{SM}(t)$ é um conjunto finito não vazio de terms normais fechados de tipo $\rho$.
\end{theorem}

%\subsection{Formulas of $\Lstar$}

For each type $\rho$, we have an equality symbol $=_\rho$ and a membership symbol $\in_\rho$. 

An \emph{atomic formula} is a formula of the form $\bot$, $t^\rho =_\rho q^\rho$ , $t^\rho \in_\rho q^{\rho^*}$ or $R(t_1^G,\ldots,t_n^G)$ where $R$ is a relational symbol of $\mathcal{L}$.

%Por diante, $A_0,B_0,\ldots$ denotam fórmulas sem quantificadores.

\begin{definition}
	The \emph{formulas} of $\Lstar$ are inductively generated as follows:
	\begin{itemize}
		\item[(i)] Atomic formulas are formulas;
		\item[(ii)] If $A$ and $B$ are formulas and $x$ is a variable, then $A \lor B, A \land B, A \rta B,\forall x \, A$ and $\exists x \, A$ are formulas.
		\item[(iii)] If $A$ and $B$ are formulas, $x$ is a variable of type $\rho$ and $t$ is a term of type $\rho^*$ where $x^{\rho}$ does not occur, then $\forall x \in t\, A, \exists x \in t\, A$ are formulas. [$\forall x \in t$ and $\exists x \in t$ are called bounded quantifications.]
	\end{itemize}
\end{definition}
%Usamos as abreviações , $A \leftrightarrow B$, $\exists x \in t\, A(x)$ e $\exists x \, A(x)$ como tipicamente na lógica clássica.

A formula is said to be \emph{$\exists$-free} if it does not contain any unbounded existential quantifiers $\exists x$.

We use the following abbreviations:

\begin{itemize}
	%\item $\bot :\equiv 0 =_0 1$
	\item[] $\lnot A :\equiv A \rightarrow \bot$
	\item[] $A \leftrightarrow B :\equiv (A \rightarrow B) \land (B \rightarrow A)$
\end{itemize}

%\subsection{Intuitionistic Logic $\ILp$}

\vspace{.3cm}

The theory $\IL$ that we are about to introduce is intuitionistic logic in all finite types. First we start by listing what we call the $\ILp$ axioms and rules (following the formalization in \cite{AvigadFeferman(98)}).

\paragraph{Axioms of $\ILp$\textup{:}}
\begin{itemize}
	\item[] $A \lor A \rightarrow A$, $A \rightarrow A \land A$
	\item[] $A \rightarrow A \lor B$, $A \land B \rightarrow A$
	\item[] $A \land B \rightarrow B \land A, A \lor B \rightarrow B \lor A$
	\item[] $\bot \rightarrow A$
	\item[] $\forall x\, A \rightarrow A[t/x]$, $A[t/x] \rightarrow \exists x\, A$, where $t$ is a term free for $x$ in $A$ and $A[t/x]$ denotes the formula obtained from $A$ after replacing each instance of $x$ by $t$.
\end{itemize}

\paragraph{Rules of $\ILp$\textup{:}}
\begin{tabular}{lll}
	\begin{tabular}{c}
		$A, A \rightarrow B$\\
		\hline
		$B$
	\end{tabular}
	&
	\begin{tabular}{c}
		$A \rightarrow B, B \rightarrow C$\\
		\hline
		$A \rightarrow C$
	\end{tabular}
	&
	\begin{tabular}{c}
		$A \rightarrow B$\\
		\hline
		$C \lor A \rightarrow C \lor B$
	\end{tabular}
\end{tabular}

\begin{tabular}{ll}
	\begin{tabular}{c}
		$A \land B \rightarrow C$\\
		\hline
		$A \rightarrow (B \rightarrow C)$
	\end{tabular}
	&
	\begin{tabular}{c}
		$A \rightarrow (B \rightarrow C)$ \\
		\hline
		$A \land B \rightarrow C$
	\end{tabular}
\end{tabular}
	
%\end{tabular}

\vspace{.3cm}

In the following rules, $x$ does not occur free in $B$.

\vspace{.3cm}

\begin{tabular}{ll}
	\begin{tabular}{c}
		$B \rightarrow A$\\
		\hline
		$B \rightarrow \forall x\, A$
	\end{tabular}
&
	\begin{tabular}{c}
		$A \rightarrow B$\\
		\hline
		$\exists x\, A \rightarrow B$
	\end{tabular}
\end{tabular}

\vspace{.3cm}
%\subsection{Axioms and rules of intuitionistic logic $\IL$}

The theory $\IL$ is an extension of $\ILp$ (in the language of all finite types) with the following axioms and axiom schemes:
\begin{itemize}
	%\item Axioms and rules of $\ILp$ %$IL^{\omega}_{-=}$
	\item[(i)] Axioms for $=_\rho$:
	\begin{itemize}
		\item[] $x =_\rho x$
		%\item $x =_\rho y \rightarrow y =_\rho x$
		%\item $x =_\rho y \land y =_\rho z \rightarrow x =_\rho z$
		\item[] $x =_\rho y \land A \rta A'$, where $A$ is an atomic formula and $A'$ is obtained from $A$ replacing some instances of $x$ by $y$.
	\end{itemize}
	
	\item[(ii)] $\forall x \in t\, A(x) \leftrightarrow \forall x (x \in t \rightarrow A(x))$, $\exists x \in t\, A(x) \leftrightarrow \exists x (x \in t \land A(x))$
	
	\item[(iii)] $\Sigma xyz= xz(yz)$, $\Pi xy = x$
	
	\item[(iv)] Axioms for star types:
	\begin{itemize}
		\item[] $w \in \mathfrak{s}x \leftrightarrow w=x$
		\item[] $w \in \cup xy \leftrightarrow w \in x \lor w \in y$
		%\item $w \in \bigcup xy \leftrightarrow \exists z \in x (w \in yz)$
		\item[] $z \in x \land w \in yz \rightarrow w \in \bigcup xy$
		%\item $x = y \leftrightarrow (w \in x \leftrightarrow w \in y)$
		\item[] $\bigcup(\mathfrak{s}x)y = yx$
		\item[] $\bigcup(\cup xy)z = \cup(\bigcup xz)(\bigcup yz)$
	\end{itemize}
\end{itemize}

\begin{remark}
	From the equalities of the conversions and the axioms for equality, we conclude that reduction implies equality.
\end{remark}

\begin{lemma} \label{setlike}
	Let $t$ be a closed term of type $\rho^*$. Then there are closed terms $q_1,\ldots,q_n$ of type $\rho$ such that
	\begin{equation*}
		\IL \vdash w\in t \leftrightarrow w=q_1 \lor \ldots \lor w=q_n
	\end{equation*}
\end{lemma}

\begin{proof}
	Since reduction implies equality and due to the axioms of equality, we only need to show the case when $t$ is in normal form, that is, when $t$ is set-like. The proof follows by induction on the complexity of $t$.
	If $t$ is $\mathfrak{s}q$, by the axiom for $\mathfrak{s}$, we can take the term $q$. If $t$ is $\cup rs$, we apply the induction hypothesis. Let $r_1,\ldots,r_m$ be the terms associated to $r$ and $s_1,\ldots,s_l$ be the terms associated to $s$. Then, by the axiom for $\cup$:
	\begin{equation*}
		\IL \vdash w\in \cup rs \leftrightarrow w=r_1 \lor \ldots \lor w=r_m \lor w=s_1\lor \ldots \lor w=s_l
	\end{equation*}
\end{proof}

%\subsubsection{Combinatorial Completeness}

Given that terms are derived solely from constants and variables through the operation of application, we naturally achieve what is commonly referred to as \emph{combinatorial completeness}. 

\begin{theorem}[Combinatorial Completeness]
	For each term $t^\sigma$ and variable $x^\rho$ there exists a term of type $\rho \rta \sigma$, denoted by $\lambda x.t$, whose variables are those of $t$ except $x$, such that for all term $s$ of type $\rho$:
	
	\begin{equation*}
		\IL \vdash (\lambda x.t)s =_\sigma t[s/x]
	\end{equation*}
	where $t[s/x]$ is the term obtained replacing in $t$ every occurrence of $x$ by $s$.
\end{theorem}

\begin{remark}
	In fact, we get $(\lambda x.t)s \succeq t[s/x]$.
\end{remark}

\begin{definition}
	A type is said to be \emph{end-star} if it is of the form $\rho_1 \rta \ldots \rta \rho_n \rta \tau^*$.
\end{definition}

For ease of reading, we denote $\mathfrak{s}x$ by $\{x\}$, $\cup xy$ by $x \cup y$ and $\bigcup xy$ by $\bigcup_{w \in x} yw$.

\vspace{.3cm}

Let $\sigma$ be the type $\rho_1 \rta \ldots \rta \rho_n \rta \tau^*$ (with $n\geq 0$) and $a,b$ be terms of type $\sigma$. The following abbreviations will also be useful:

%\begin{definition}
%	Let $\sigma$ be the type $\rho_1 \rta \ldots \rta \rho_n \rta \tau^*$. Given terms $a,b$ of type $\sigma$, we define the following abbreviations:%\begin{itemize}
\begin{equation*}
	a \sqsubseteq_{\sigma} b :\equiv \forall x_1^{\rho_1},\ldots,x_n^{\rho_n} (ax_1\cdots x_n \subseteq_{\tau^*} bx_1\cdots x_n)
\end{equation*}
where $a' \subseteq_{\tau^*} b' :\equiv \forall x (x \in a' \rta x \in b')$ for $a',b'$ of type $\tau^*$.

\begin{equation*}
	a \sqcup_{\sigma} b :\equiv \lambda x_1^{\rho_1},\ldots,x_n^{\rho_n}.ax_1\cdots x_n \cup bx_1\cdots x_n
\end{equation*}
\begin{equation*}
	\bigsqcup_{w \in F} fw :\equiv \lambda x_1^{\rho_1},\ldots,x_n^{\rho_n}.\bigcup_{w \in F} fwx_1\cdots x_n
\end{equation*}
%\end{itemize}

$\ux \sqsubseteq \ux'$ denotes $\bigwedge_i x_i \sqsubseteq x'_i$, where $\ux$ (respectively $\ux'$) denotes the tuple of $x_i$ (respectively $x'_i$).
%\end{definition}

\begin{remark}
	We have $a \sqsubseteq_{\sigma} a \sqcup_{\sigma} b, b \sqsubseteq_{\sigma} a \sqcup_{\sigma} b$ and $\forall z \in F (fz \sqsubseteq_{\sigma} \bigsqcup_{w \in F} fw)$.
\end{remark}

\begin{remark}
	The relation $\sqsubseteq_{\tau^*}$ is simply $\subseteq_{\tau^*}$.
\end{remark}

\section{The herbrandized modified realizability}\label{hmr}

In this section, we define the new realizability notion (the \emph{herbrandized modified realizability}) within $\IL$ and prove the corresponding soundness and characterization theorems.

\begin{definition}
	Given a formula $A$, we define $A^{HR}$ as a formula, with the free variables of $A$, of the form:
	\begin{equation*}
		A^{HR} \equiv \exists \ux \,A_{HR}(\ux)
	\end{equation*}
	
	\noindent where $A_{HR}$ is a $\exists$-free formula. The definition is by induction on the complexity of $A$:
	
	If $A$ is atomic, $A^{HR} :\equiv A_{HR} :\equiv A$.
	
	If $A^{HR} \equiv \exists \ux \,A_{HR}(\ux)$ and $B^{HR} \equiv \exists \uu \,B_{HR}(\uu)$:
	
	\begin{itemize}
		\item[] $(A \lor B)^{HR} :\equiv \exists \ux,\uu (A_{HR}(\ux) \lor B_{HR}(\uu))$
		\item[] $(A \land B)^{HR} :\equiv \exists \ux,\uu (A_{HR}(\ux) \land B_{HR}(\uu))$
		\item[] $(A \rightarrow B)^{HR} :\equiv \exists \uU \forall \ux (A_{HR}(\ux) \rightarrow B_{HR}(\uU\ux))$
		%\item $(\lnot A)^{HR} :\equiv \forall \uY \exists \ux' [\exists \ux \in \ux' \lnot A_{HR}(\ux,\uY\ux)]$
		
		\item[] $(\exists z \,A(z))^{HR} :\equiv \exists Z,\ux \exists z \in Z \,A_{HR}(z,\ux)$
		\item[] $(\forall z \,A(z))^{HR} :\equiv \exists \uX \forall z \,A_{HR}(z,\uX z)$
		\item[] $(\exists z \in t\,A(z))^{HR} :\equiv \exists \ux \exists z \in t\,A_{HR}(z,\ux)$
		\item[] $(\forall z \in t\,A(z))^{HR} :\equiv \exists \ux \forall z \in t\,A_{HR}(z,\ux)$
		
	\end{itemize}
	
\end{definition}

\begin{remark}
	If $A$ is an $\exists$-free formula, then $A^{HR} \equiv A_{HR} \equiv A$.
\end{remark}

\begin{remark}
	By induction on the logical structure of the formulas, it can be shown that, for any formula $A$, the variables $\ux$ in $A^{HR} \equiv \exists \ux \,A_{HR}(\ux)$ are of end-star type.
\end{remark}

\begin{theorem}[Monotonicity]
	For any formula $A$ with $A^{HR} \equiv \exists \ux \,A_{HR}(\ux)$:
	\begin{equation*}
		\ux \sqsubseteq \ux'  \land A_{HR}(\ux) \rta  A_{HR}(\ux')
	\end{equation*}
\end{theorem}

\begin{proof}
	The proof is by induction on the logical structure of the formula. For atomic formulas the result is trivial. The induction step in the cases $A \land B$, $A \lor B$, $\exists z \in t\,A(z)$ and $\forall z \in t\,A(z)$ follows from the induction hypothesis.
	
	\begin{itemize}
		\item[]Case $A \to B$
		
		We assume $\uU \sqsubseteq \uU'$ and $\forall \ux (A_{HR}(\ux) \rightarrow B_{HR}(\uU\ux))$. We need to show that $\forall \ux (A_{HR}(\ux) \rightarrow B_{HR}(\uU'\ux))$. Take $\ux$ such that $A_{HR}(\ux)$. By hypothesis, we get $B_{HR}(\uU\ux)$. We note that from $\uU \sqsubseteq \uU'$ we get $\uU\ux \sqsubseteq \uU'\ux$. By induction hypothesis, we have that $B_{HR}(\uU'\ux)$. Thus $\forall \ux (A_{HR}(\ux) \rightarrow B_{HR}(\uU'\ux))$.
		
		\item[]Case $\forall z\, A(z)$
		
		We assume $\uX \sqsubseteq \uX'$ and $\forall z \,A_{HR}(z,\uX z)$. We need to show $\forall z \,A_{HR}(z,\uX' z)$. This follows from the induction hypothesis, noting that $\uX \sqsubseteq \uX'$ implies $\uX z \sqsubseteq \uX'z$.
		
		\item[]Case $\exists z\, A(z)$
		
		We assume $Z \sqsubseteq Z', \ux \sqsubseteq \ux'$ and $\exists z \in Z \,A_{HR}(z,\ux)$. We need to show $\exists z \in Z' \,A_{HR}(z,\ux')$. Take $z$ such that $z\in Z$ and $A_{HR}(z,\ux)$. Since $Z \sqsubseteq Z'$, we have that $z\in Z'$. From the induction hypothesis we get $A_{HR}(z,\ux')$. Thus $\exists z \in Z' \,A_{HR}(z,\ux')$.
	\end{itemize}
	
\end{proof}

%\begin{remark}
%	Note that, in the case of the bounded modified realizability $(\cdot)^{br}$ \cite{FerreiraNunes(06)}, the translation of the universal quantification needs to be more evolved, namely
%	\begin{equation*}
%		(\forall z A(z))^{br}:\equiv \widetilde{\exists} \uX \widetilde{\forall} a \forall z\leq^* a A_{br}(z,\uX a)
%		\end{equation*}
%	Monotonicity explains this need, because one has to ensure that 
%	\begin{equation*}
%		X\leq^* Y \to Xz\leq^* Yz
%	\end{equation*}
%which is the case if $z$ is monotone, i.e., if $z\leq^* z $, thus the need to consider monotone majorants $a$ such that $z\leq^* a$.
%In the case of the herbrand modified realizability we can deal with $z$ itself since 
%\begin{equation*}
%X \sqsubseteq Y \to Xz \sqsubseteq Yz
%\end{equation*}	
%\end{remark}

Consider the following Choice and Independence of Premises Principles:
\begin{equation*}
	\AC\text{: } \forall x^\rho \exists y^\sigma \,A(x,y) \rta \exists f^{\rho \rta \sigma^*} \forall x \exists y \in fx \,A(x,y)
\end{equation*}
and
\begin{equation*}
	\mathrm{IP}^*_{\not\exists} \text{: } (B(x) \rightarrow \exists y \,A(y)) \rightarrow \exists w (B(x) \rightarrow \exists y \in w \,A(y))
\end{equation*}
where $B$ is an $\exists$-free formula.

%\begin{obs}
% $\mathrm{bCColl}$ and $\mathrm{LLPO}^*$ are $\exists$-free formulas.
%\end{obs}

The theory $\IL+ \AC + \mathrm{IP}^*_{\not\exists} $ proves the following Collection Principle:

\begin{lemma}[Collection] \label{Coll}
	\begin{equation*}
		\IL+ \AC + \mathrm{IP}^*_{\not\exists} \vdash \forall \ux \in \uy \exists \uz\, A(\ux,\uz) \rta \exists \uw \forall \ux \in \uy \exists \uz \in \uw\, A(\ux,\uz)
	\end{equation*}
\end{lemma}

\begin{theorem}[Soundness]
	Let $A$ be a formula with free variables $\ua$. Let $T_{\not\exists}$ be a set of $\exists$-free formulas. If
	\begin{equation*}
		\IL + \AC + \mathrm{IP}^*_{\not\exists} + T_{\not\exists} \vdash A(\ua)
	\end{equation*}
	
	then there exist closed terms $\ut$ such that
	\begin{equation*}
		\IL +T_{\not\exists} \vdash  A_{HR}(\ua,\ut\ua)
	\end{equation*}
\end{theorem}

\begin{proof}\label{sound}
	Fix $A^{HR} \equiv \exists \ux \,A_{HR}(\ux)$, $B^{HR} \equiv \exists \uu \,B_{HR}(\uu)$ and $C^{HR} \equiv \exists \up\,C_{HR}(\up)$.
	
	We use the tuple notation combined with the lambda notation as per the following example: if $\ux$ denotes the tuple of $x_i$, then $\lambda \ux,\uy.\ux$ denotes the tuple of $\lambda \ux,\uy.x_i$.
	
	The proof is by induction on the deduction of the formula.
	
	\begin{itemize}
		
		\item[] Case $A \rta A \land A$
		
		We need terms $\ut_{\uX'},\ut_{\uX''}$ such that:
		\begin{equation*}
			\forall \ux (A_{HR}(\ux) \rta A_{HR}(\ut_{\uX'}\ux) \land A_{HR}(\ut_{\uX''}\ux))
		\end{equation*}
		
		We take $\ut_{\uX'}:=\ut_{\uX''}:=\lambda \ux.\ux$.
		
		\item[] Case $A \lor A \rta A$
		
		We need terms $\ut_{\uX''}$ such that:
		\begin{equation*}
			\forall \ux,\ux' (A_{HR}(\ux) \lor A_{HR}(\ux') \rta A_{HR}(\ut_{\uX''}\ux\ux'))
		\end{equation*}
		
		We take $\ut_{\uX''}:=\lambda \ux,\ux'.\ux \sqcup \ux'$.
		
		\item[] Case $A \land B \rta A$
		
		We need terms $\ut_{\uX'}$ such that:
		\begin{equation*}
			\forall \ux,\uu (A_{HR}(\ux) \land B_{HR}(\uu) \rta A_{HR}(\ut_{\uX'}\ux\uu))
		\end{equation*}
		
		We take $\ut_{\uX'}:=\lambda \ux,\uu.\ux$.
		
		\item[] Case $A \rta A \lor B$
		
		We need terms $\ut_{\uX'},\ut_{\uU}$ such that:
		\begin{equation*}
			\forall \ux (A_{HR}(\ux) \rta A_{HR}(\ut_{\uX'}\ux) \lor B_{HR}(\ut_{\uU}\ux))
		\end{equation*}
		
		We take $\ut_{\uX'}:= \lambda \ux.\ux$. For $\ut_{\uU}$ we take an arbitrary closed term of adequate type, which exists due to combinatorial completeness and the existence of a constant of type $G$.
		
		\item[] Case $A \land B \rta B \land A$
		
		We need terms $\ut_{\uX'},\ut_{\uU'}$ such that:
		\begin{equation*}
			\forall \ux,\uu (A_{HR}(\ux) \land B_{HR}(\uu) \rta B_{HR}(\ut_{\uU'}\ux\uu) \land A_{HR}(\ut_{\uX'}\ux\uu))
		\end{equation*}
		
		We take $\ut_{\uX'}:= \lambda \ux,\uu.\ux, \quad\ut_{\uU'}:= \lambda \ux,\uu.\uu$.
		
		\item[] The case $A \lor B \rta B \lor A$ is analogous to the above.
		
		\item[] Case $\bot \rta A$
		
		We need terms $\ut_{\ux}$ such that:
		\begin{equation*}
			\bot \rta A_{HR}(\ut_{\ux})
		\end{equation*}
		
		We take arbitrary closed $\ut_{\ux}$.
		
		\item[] Case $\forall z\, A \rta A(t)$
		
		We need terms $\ut_{\uX''}$ such that:
		\begin{equation*}
			\forall \uX' (\forall z\, A_{HR}(z,\uX' z) \rta A_{HR}(t,\ut_{\uX''}\uX'))
		\end{equation*}
		
		We take $\ut_{\uX''}:=\lambda \uX'.\uX't$.
		
		\item[] Case $A(t) \rta \exists z\, A$
		
		We need terms $t_{Z'},\ut_{\uX'}$ such that:
		\begin{equation*}
			\forall \ux (A_{HR}(t,\ux) \rta \exists z\in t_{Z'}\ux\, A_{HR}(z,\ut_{\uX'}\ux))
		\end{equation*}
		
		We take $t_{Z'}:=\lambda \ux.\{t\}, \quad \ut_{\uX'}:= \lambda \ux.\ux$.
		
		\item[] Case
		\begin{tabular}{c}
			$A, A \rightarrow B$\\
			\hline
			$B$
		\end{tabular}
		
		By induction hypothesis, we obtain closed terms $\ut_{\ux},\ut_{\uU}$ such that:
		\begin{equation*}
			A_{HR}(\ut_{\ux})
		\end{equation*}
		\begin{equation*}
			\forall \ux (A_{HR}(\ux) \rightarrow B_{HR}(\ut_{\uU}\ux))
		\end{equation*}
		
		We need terms $\ut_{\uu}$ such that
		\begin{equation*}
			B_{HR}(\ut_{\uu})
		\end{equation*}
		
		We take $\ut_{\uu} := \ut_{\uU}\ut_{\ux}$.
		
		\item[] Case
		\begin{tabular}{c}
			$A \rightarrow B, B \rightarrow C$\\
			\hline
			$A \rightarrow C$
		\end{tabular}
		
		By induction, we have closed terms $\ut_{\uU},\ut_{\uP}$ such that:
		\begin{equation*}
			\forall \ux (A_{HR}(\ux) \rightarrow B_{HR}(\ut_{\uU}\ux))
		\end{equation*}
		\begin{equation*}
			\forall \uu (B_{HR}(\uu) \rightarrow C_{HR}(\ut_{\uP}\uu))
		\end{equation*}
		
		We need terms $\us_{\uP}$ such that:
		\begin{equation*}
			\forall \ux (A_{HR}(\ux) \rightarrow C_{HR}(\us_{\uP}\ux))
		\end{equation*}
		
		We take $\us_{\uP} := \lambda \ux.\ut_{\uP}(\ut_{\uU}\ux)$.
		
		\item[] Cases
		\begin{tabular}{c}
			$A \land B \rightarrow C$\\
			\hline
			$A \rightarrow (B \rightarrow C)$
		\end{tabular}
		,
		\begin{tabular}{c}
			$A \rightarrow (B \rightarrow C)$ \\
			\hline
			$A \land B \rightarrow C$
		\end{tabular}
		
		To realize $A \land B \rta C$ we need closed terms $\ut_{\uP}$ such that:
		\begin{equation*}
			\forall \ux,\uu (A_{HR}(\ux) \land B_{HR}(\uu) \rta C_{HR}(\ut_{\uP}\ux\uu))
		\end{equation*}
		
		To realize $A \rta (B \rta C)$ we need closed terms $\ut_{\uP}$ such that:
		\begin{equation*}
			\forall \ux (A_{HR}(\ux) \rta \forall \uu(B_{HR}(\uu) \rta C_{HR}(\ut_{\uP}\ux\uu)))
		\end{equation*}
		
		These two formulas are equivalent in $\IL$, therefore the terms which realize $A \land B \rightarrow C$ are exactly those which realize $A \rightarrow (B \rightarrow C)$.
		
		\item[] Case
		\begin{tabular}{c}
			$A \rightarrow B$\\
			\hline
			$C \lor A \rightarrow C \lor B$
		\end{tabular}
		
		By induction we have terms $\ut_{\uU}$ such that
		\begin{equation*}
			\forall \ux (A_{HR}(\ux) \rightarrow B_{HR}(\ut_{\uU}\ux))
		\end{equation*}
		
		We need terms $\ut_{\uP'},\ut_{\uU'}$ such that
		\begin{equation*}
			\forall \up,\ux (C_{HR}(\up) \lor A_{HR}(\ux) \rta C_{HR}(\ut_{\uP'}\up\ux) \lor B_{HR}(\ut_{\uU'}\up\ux)
		\end{equation*}
		
		We take $\ut_{\uP'}:=\lambda \up,\ux.\up,\quad\ut_{\uU'}:= \lambda \up,\ux.\ut_{\uU}\ux$.
		
		\item[] Case
		\begin{tabular}{c}
			$B \rightarrow A$\\
			\hline
			$B \rightarrow \forall z\, A$
		\end{tabular}
		where $z$ is not free in $B$
		
		By induction, we have terms $\ut_{\uX}$ such that
		\begin{equation*}
			\forall \uu (B_{HR}(\uu) \rta A_{HR}(z,\ut_{\uX}z\uu ))
		\end{equation*}
		
		We need $\us_{\uX}$ such that
		\begin{equation*}
			\forall \uu (B_{HR}(\uu) \rta \forall z\, A_{HR}(z,\us_{\uX}\uu z))
		\end{equation*}
		
		We take $\us_{\uX} := \lambda \uu, z.\ut_{\uX}z\uu$.
		
		\item[] Case
		\begin{tabular}{c}
			$A \rightarrow B$\\
			\hline
			$\exists z\, A \rightarrow B$
		\end{tabular}
		where $z$ is not free in $B$
		
		By induction we have terms $\ut_{\uU}$ such that
		\begin{equation*}
			\forall \ux (A_{HR}(z,\ux) \rta B_{HR}(\ut_{\uU}z\ux))
		\end{equation*}
		
		We need terms $\us_{\uU}$ such that
		\begin{equation*}
			\forall z',\ux (\exists z \in z'\, A_{HR}(z,\ux) \rta B_{HR}(\us_{\uU}z'\ux))
		\end{equation*}
		
		We take $\us_{\uU} := \lambda z',\uu.\bigsqcup_{z\in z'}\ut_{\uU}z\ux$.
		
		\item[] The axioms for $=$, $\Pi$ and $\Sigma$, as well as the axioms for the star types, are $\exists$-free formulas and therefore don't need terms.
		
		\item[] Case $\forall z\in t\, A	(z) \rta \forall z (z \in t \rta A(z))$
		
		We need terms $\ut_{\uX'}$ such that
		\begin{equation*}
			\forall \ux (\forall z\in t\, A_{HR}(z,\ux) \rta \forall z(z\in t \rta A_{HR}(z,\ut_{\uX'}\ux z)))
		\end{equation*}
		
		We take $\ut_{\uX'} := \lambda \ux,z.\ux$.
		
		\item[] Case $\forall z\in t\, A	(z) \leftarrow \forall z (z \in t \rta A(z))$
		
		We need terms $\ut_{\uX}$ such that
		\begin{equation*}
			\forall \uX' (\forall z (z\in t \rta A_{HR}(z,\uX'z)) \rta \forall z\in t\, A_{HR}(z,\ut_{\uX}\uX'))
		\end{equation*}
		
		We take $\ut_{\uX} := \lambda\uX'.\bigsqcup_{z\in t} \uX'z$.
		
		\item[] Case $\exists z \in t\, A(z) \rta \exists z (z \in t \land A(z))$
		
		We need terms $\ut_{\uX}$ and $t_{Z}$ such that
		\begin{equation*}
			\forall \ux (\exists z \in t \, A_{HR}(z,\ux) \rta \exists z \in t_{Z}\ux (z\in t \land A_{HR}(z,\uX\ux)))
		\end{equation*}
		
		We take $t_{Z} := \lambda \ux.t,\quad \ut_{\uX} := \lambda \ux.\ux$.
		
		\item[] Case $\exists z \in t\, A(z) \leftarrow \exists z (z \in t \land A(z))$
		
		We need terms $\ut_{\uX}$ such that
		\begin{equation*}
			\forall z',\ux (\exists z \in z' (z\in t \land A_{HR}(z,\ux)) \rta \exists z \in t\, A_{HR}(z,\ut_{\uX} z'\ux))
		\end{equation*}
		
		We take $\ut_{\uX} := \lambda z',\ux.\ux$.
		
		\item[] Case $\AC\text{: } \forall z \exists w \,A(z,w) \rta \exists W \forall z \exists w \in Wz \,A(z,w)$
		
		We need terms $t_{W''}, \ut_{\uX''}$ such that:
		\begin{multline*}
			\forall W',\uX' (\forall z \exists w \in W'z\, A_{HR}(z,w,\uX'z) \rta\\ \exists W \in t_{W''}W'\uX' \forall z' \exists w' \in Wz'\, A_{HR}(z',w',\ut_{\uX''}W'\uX'z'))
		\end{multline*}
		
		We take $t_{W''} := \lambda W',\uX'.\{W'\}, \quad \ut_{\uX''} := \lambda W',\uX'.\uX'$.
		
		\item[] Case $\mathrm{IP}^*_{\not\exists} \text{: } (A \rightarrow \exists w \,B(w)) \rightarrow \exists z (A \rightarrow \exists w \in z \,B(w))$ 
		where $A$ is an $\exists$-free formula.
		
		We need terms $t_{Z'},\ut_{\uU}$ such that
		\begin{equation*}
			\forall  w' ,\uu ((A \rta \exists w \in  w' \, B_{HR}(w,\uu)) \rta \exists z \in t_{Z'} w' \uu (A \rta \exists w \in z\, B_{HR}(w,\ut_{\uU} w' \uu)))
		\end{equation*}
		
		We take $t_{Z'}:= \lambda  w' ,\uu.\{ w' \},\quad \ut_{\uU}:= \lambda  w' ,\uu.\uu$.
		
		\item[] The formulas in $T_{\not\exists}$ verify themselves.
	\end{itemize}
\end{proof}

\begin{theorem}[Characterisation Theorem] \label{carTh}
	For any formula $A$:
	\begin{equation*}
		\IL + \AC + \mathrm{IP}^*_{\not\exists} \vdash A \leftrightarrow A^{HR}
	\end{equation*}
\end{theorem}
\begin{proof}
	%\subsection{Proof of Characterisation Theorem} \label{car}

	The proof is by induction on the logical structure of the formula. The result is immediate for $\exists$-free formulas. For the induction step, we use the induction hypothesis immediately, for instance, in order to show $A \land B \leftrightarrow (A \land B)^{HR}$ we show that $A^{HR} \land B^{HR} \leftrightarrow (A \land B)^{HR}$.
	
	The equivalence for the cases $A \land B$, $A \lor B$, $\exists z\, A(z)$ and $\exists z \in t\, A(z)$ is evidently a consequence of the logical rules; for instance, for $A \land B$, we just drag the quantifiers inside and outside.
	
	For the remaining cases, the right to left implication is also consequence of the logical rules.
	
	We focus on the reciprocal.
	
	\begin{itemize}
		\item[] Case $A \rta B$
		
		We are going to show that $(\exists \ux\, A_{HR}(\ux) \rta \exists \uu\, B_{HR}(\uu))\rta (\exists \uU'''\forall \ux (A_{HR}(\ux) \rta B_{HR}(\uU'''\ux)))$.
		
		From $\exists \ux\, A_{HR}(\ux) \rta \exists \uu\, B_{HR}(\uu)$ we have $\forall \ux (A_{HR}(\ux) \rta \exists \uu\, B_{HR}(\uu))$. Using $\mathrm{IP}^*_{\not\exists}$ we obtain $\forall \ux \exists \uu' (A_{HR}(\ux) \rta \exists \uu\in \uu'\, B_{HR}(\uu))$. By monotonicity, we get $\forall \ux \exists \uu' ( A_{HR}(\ux) \rta B_{HR}(\bigsqcup_{\uu \in \uu'}\uu))$, i.e., $\forall \ux \exists \uu'' ( A_{HR}(\ux) \rta  B_{HR}(\uu''))$. Applying $\AC$, we get $\exists \uU'' \forall \ux \exists \uu'' \in \uU''\ux (A_{HR}(\ux) \rta B_{HR}(\uu''))$. By monotonicity, we obtain $\exists \uU'' \forall \ux (A_{HR}(\ux) \rta B_{HR}(\bigsqcup_{\uu'' \in \uU''\ux}\uu''))$, and therefore $\exists \uU'''\forall \ux (A_{HR}(\ux) \rta B_{HR}(\uU'''\ux))$.
		
		\vspace{.3cm}
		
		\item[] Case $\forall z\, A(z)$
		
		We are going to show that $\forall z \exists \ux \, A_{HR}(z,\ux) \rta \exists \uX' \forall z\, A_{HR}(z,\uX' z)$.
		
		From $\forall z \exists \ux \, A_{HR}(z,\ux)$, applying $\AC$ we obtain $\exists \uX \forall z\exists \ux \in \uX z\, A_{HR}(z,\ux)$. 
		By monotonicity we have $\exists \uX \forall z\, A_{HR}(z,\bigsqcup_{\ux \in \uX z}\ux)$. Thus $\exists \uX' \forall z\, A_{HR}(z,\uX' z)$.
		
		\vspace{.3cm}
		
		\item[] Case $\forall z\in t\, A(z)$
		
		We are going to show that $\forall z\in t\exists \ux \, A_{HR}(z,\ux) \rta \exists \ux \forall z\in t\, A_{HR}(z,\ux)$.
		
		Using collection (Lemma \ref{Coll}, involving $\AC$ and $\mathrm{IP}^*_{\not\exists}$), we get $\exists \uw \forall z\in t \exists \ux \in \uw \, A_{HR}(z,\ux)$. Take $\ux':=\bigsqcup_{\ux\in \uw}\ux$. Then, by monotonicity, we get $A_{HR}(z,\ux')$ for all $z\in t$.
	\end{itemize}
\end{proof}

%\pagebreak

\begin{corollary} \label{cor1}
	If
	\begin{equation*}
		\IL + \AC + \mathrm{IP}^*_{\not\exists} + T_{\not\exists} \vdash \forall \uz \exists \uw \, A(\uz,\uw)
	\end{equation*}
	
	\noindent	where $\uz,\uw$ are the only free variables of $A$, then there exist closed terms $\ur$ such that
	\begin{equation*}
		\IL + \AC + \mathrm{IP}^*_{\not\exists} + T_{\not\exists} \vdash \forall \uz \exists \uw \in \ur\uz\, A(\uz,\uw)
	\end{equation*}
\end{corollary}

\begin{proof}
	We have:
	\begin{equation*}
		(\forall \uz \exists \uw \, A(\uz,\uw))^{HR} \equiv \exists \uW', \uX' \forall \uz \exists \uw \in \uW'\uz\, A_{HR}(\uz,\uw,\uX'\uz)
	\end{equation*}
	
	Using Soundness, we get terms $\ut,\ur$ such that, for arbitrary $\uz$
	\begin{equation*}
		\exists \uw \in \ur\uz\, A_{HR}(\uz,\uw,\ut \uz)
	\end{equation*}
	
	thus
	\begin{equation*}
		\exists \ux \exists \uw \in \ur\uz\, A_{HR}(\uz,\uw,\ux)
	\end{equation*}
	
	which is just $(\exists \uw \in \ur\uz\, A(\uz,\uw))^{HR}$. By the Characterization Theorem, we obtain $\exists \uw \in \ur\uz\, A(\uz,\uw)$ for arbitrary $\uz$.
\end{proof}

\begin{corollary}\label{cor2}
	If
	\begin{equation*}
		\IL + \AC + \mathrm{IP}^*_{\not\exists} + T_{\not\exists} \vdash \forall \uz \exists \uw \, A(\uz,\uw)
	\end{equation*}
	
	\noindent where $A$ is an $\exists$-free formula whose only free variables are $\uz,\uw$, then there exist closed terms $\ur$ such that
	\begin{equation*}
		\IL + T_{\not\exists} \vdash \forall \uz \exists \uw \in \ur\uz\, A(\uz,\uw)
	\end{equation*}
\end{corollary}

\begin{proof}
	We have:
	\begin{equation*}
		(\forall \uz \exists \uw \, A(\uz,\uw))^{HR} \equiv \exists \uW' \forall \uz \exists \uw \in \uW'\uz\, A(\uz,\uw)
	\end{equation*}
	
	Using Soundness, we get terms $\ur$ such that
	\begin{equation*}
		\IL + T_{\not\exists} \vdash \forall \uz \exists \uw \in \ur\uz\, A(\uz,\uw)
	\end{equation*}
\end{proof}

\begin{remark}
	We also have the previous results in the absence of the variables $\uz$. Due to Lemma \ref{setlike}, we conclude the following:
\end{remark}

\begin{corollary}\label{cor3}
	If
	\begin{equation*}
		\IL + \AC + \mathrm{IP}^*_{\not\exists} + T_{\not\exists} \vdash \exists x \, A(x)
	\end{equation*}
	
	\noindent where $x$ is the only free variable of $A$, then there exist closed terms $t_1,\ldots,t_n$ such that
	\begin{equation*}
		\IL + \AC + \mathrm{IP}^*_{\not\exists} + T_{\not\exists} \vdash A(t_1) \lor \ldots \lor A(t_n)
	\end{equation*}
\end{corollary}

\begin{corollary}\label{cor4}
	If
	\begin{equation*}
		\IL + \AC + \mathrm{IP}^*_{\not\exists} + T_{\not\exists} \vdash \exists x \, A(x)
	\end{equation*}
	
\noindent	where $A$ is an $\exists$-free formula whose only free variable is $x$, then there exist closed terms $t_1,\ldots,t_n$ such that
	\begin{equation*}
		\IL + T_{\not\exists} \vdash A(t_1) \lor \ldots \lor A(t_n)
	\end{equation*}
\end{corollary}

%\pagebreak

%\pagebreak

\section{Final Notes}\label{section:arithmetic}

\subsection{Atomic versus $\exists$-free} 

The axiom scheme for equality, on page 5, was formulated with $A$ an atomic formula. As is well known, this atomic formulation is enough to ensure we have the result for any formula $A$ within the system. Instead of using atomic formulas, we could have presented the scheme with $A$ an $\exists$-free formula. From our interpretation perspective, the $\exists$-free formulas are still computationally empty.   

\subsection{The translation of the universal quantification} 

Note that, in the case of the bounded modified realizability $(\cdot)^{br}$ \cite{FerreiraNunes(06)}, the translation of the universal quantification needs to be more evolved, namely
\begin{equation*}
	(\forall z A(z))^{br}:\equiv \widetilde{\exists} \uX \widetilde{\forall} a \forall z\leq^* a A_{br}(z,\uX a)
\end{equation*}
Monotonicity explains this need, because one has to ensure that 
\begin{equation*}
	X\leq^* Y \to Xz\leq^* Yz
\end{equation*}
which is the case if $z$ is monotone, i.e., if $z\leq^* z $, thus the need to consider monotone majorants $a$ such that $z\leq^* a$.
In the case of the herbrand modified realizability we can deal with $z$ itself since 
\begin{equation*}
	X \sqsubseteq Y \to Xz \sqsubseteq Yz
\end{equation*}

\subsection{The star calculus on arithmetic}\label{star_arit}

The possibility of extending the star calculus to the arithmetic setting keeping the strong normalization and the Church-Rosser properties was already discussed in \cite{Ferreira(20)} but without proofs of such results. We present here such proofs. 

%It is possible to extend the star combinatory calculus presented in Section \ref{star} to the arithmetical context. 

We denote the ground type as $N$, intending it to represent natural numbers; and we consider the following constants from the language of arithmetic: $0$ of type $N$, $S$ of type $N \to N$ and $R_{\sigma}$ of type $N \to \sigma \to (\sigma \to N \to \sigma) \to \sigma$. Note that $\sigma$ can be of star type.

%We also add the constant:

%\begin{itemize}
%\item $M$ of type $N^* \rightarrow N$
%$R_{\sigma}$ of type $N \rta \sigma \rta (\sigma \rta N \rta \sigma) \rta \sigma$.
%\end{itemize}
%This constant allows us to define functions by recursion.
%They have the intended meaning:

%\item $Mt$ is the maximum of the set $t$ of natural numbers 

We assign the following conversions to the new constants above:
\begin{itemize}
	\item[] $R_{\sigma}0qr \conv q$ ($q: \sigma,\ r: \sigma \rta N \rta \sigma$)
	\item[] $R_{\sigma}(St)qr \conv r(Rtqr,t)$ ($t: N,\ q: \sigma,\ r: \sigma \rta N \rta \sigma$)
	%\item $M(\cup qr) \rightsquigarrow \max(M(q),M(r))$ ($q: N^*,\ r: N^*$)
	%\item $M(\mathfrak{s}q) \rightsquigarrow q$ ($q: N$)
\end{itemize}

%Note that $\max$ is well defined and built using $R_N$.

We are going to show, adapting the proof in \cite{FerreiraFerreira(17)} based on Tait's \emph{reducibility} technique \cite{Tait(67),Troelstra(73)}, that the star calculus in the extended (arithmetical) context remains strongly normalizable. %The structure of the proof is exactly the one in \cite{FerreiraFerreira(17)}, with the following definitions and lemmas. 
First some definitions.

\begin{definition}
	Given a term $t$ of type $\sigma^*$, we define a finite set of terms of type $\sigma$, the \emph{surface elements} of $t$, which we denote by $\mathrm{SM}(t)$, by induction on the complexity of $t$ as follows:
	
	\begin{itemize}
		\item[(i)] If $t$ is of the form $\mathfrak{s}(q)$ then $\mathrm{SM}(t)$ is $\{q\}$;
		\item[(ii)] If $t$ is of the form $\cup qr$, then $\mathrm{SM}(t)$ is $\mathrm{SM}(q) \cup \mathrm{SM}(r)$;
		\item[(iii)] In any other case, $\mathrm{SM}(t)$ is $\emptyset$.
	\end{itemize}
\end{definition}

\begin{definition}
	The set $\Red_{\sigma}$ of the \emph{reducible} terms of type $\sigma$ is defined recursively in the complexity of the type $\sigma$ as follows:
	
	\begin{itemize}
		\item[(i)] $t \in \Red_N:\equiv$ $t$ is strongly normalizable;
		\item[(ii)] $t \in \Red_{\sigma \rightarrow \tau}:\equiv$ for all $q$, if $q \in \Red_{\sigma}$ then $tq \in \Red_{\tau}$;
		\item[(iii)] $t \in \Red_{\sigma^*}:\equiv$  $t$ is strongly normalizable and if, for any term obtained by reduction of $t$, its surface elements are reducible.
	\end{itemize}
\end{definition}

The proofs of lemmas 1, 2, 3 and 4 presented in \cite{FerreiraFerreira(17)} on pages 526--527, are already applicable in the new context. We just recall here the statement of two of those results: 

%\begin{lemma} \label{l1} 
%	Let $x$ be a variable of type $\sigma_1 \rta \cdots \rta \sigma_n \rta \rho$, with $\rho$ star type or $N$. If $t_1,\ldots,t_n$ are strongly normalizable terms of types $\sigma_1,\ldots,\sigma_n$ resp., then $xt_1\cdots t_n \in \Red_\rho$.
%\end{lemma}

%\begin{proof}
%We have immediately that $xt_1\cdots t_n$ is strongly normalizable. It is enough to show, only for the star type case, that the surface elements of any term obtained by reduction of $xt_1\cdots t_n$ are reducible. One such term $q$ has the form $xt'_1\cdots t'_n$ with $t_i \succeq t'_i$, therefore $\mathrm{SM}(q)=\emptyset$. Thus $xt_1\cdots t_n \in \\Red_\rho$.
%\end{proof}

\begin{lemma} \label{l2}
	If $t \in \Red_{\sigma}$, then $t$ is strongly normalizable;
\end{lemma}

%\begin{lemma} \label{l2}
%	We have:
%	\begin{itemize}
%		\item[(i)] If $t \in \Red_{\sigma}$, then $t$ is strongly normalizable;
%		\item[(ii)] $x^{\sigma} \in \Red_{\sigma}$.
%	\end{itemize}
%\end{lemma}

%\begin{proof}
%We show both propositions simultaneously, by induction on the complexity of $\sigma$. The case where $\sigma$ is $N$ is immediate by the definition of reducible.
%Let $t \in \\Red_{\tau \rta \theta}$. By induction hypothesis, $x^\tau \in \\Red_\tau$. Therefore $tx \in \\Red_\theta$. By induction hypothesis, $tx$ is strongly normalizable. Thus $t$ is strongly normalizable.
%Now we show that $x^{\tau \rta \theta} \in \\Red_{\tau \rta \theta}$. Let $t \in \\Red_\theta$. Assume that $\theta$ is $\sigma_1 \rta \cdots \rta \sigma_n \rta \rho$, with $\rho$ star type or $N$. Sejam $t_1,\ldots,t_n$ termos reducible de tipos $\sigma_1,\ldots,\sigma_n$ resp.. Por hipótese de indução, $t,t_1,\ldots,t_n$ são strongly normalizable. Pelo lema \ref{l1}, $xtt_1\cdots t_n \in \\Red_\rho$. Logo $x \in \\Red_{\tau \rta \theta}$.
%No caso de $t$ ser tipo $\rho^*$, segue-se da definição que é strongly normalizable. A partir de redução de $x^{\rho^*}$ só se obtém o próprio $x^{\rho^*}$ e $\mathrm{SM}(x^{\rho^*})=\emptyset$, logo $x \in \\Red_{\rho^*}$.
%\end{proof}

\begin{lemma} \label{l3}
	If $t$ is a reducible term and $t \succeq q$, then $q$ is reducible.
\end{lemma}

%\begin{proof}
%Assumimos que $t$ é reducible e que $t \succeq q$. Mostramos por indução no tipo de $t$. Se for de tipo $N$, é trivial. Assumimos que é tipo $\tau \rta \sigma$, se $r \in \\Red_\tau$, $tr \succeq qr$ e $tr \in \\Red_\sigma$. Por hipótese de indução, $qr \in \\Red_\sigma$. Por definição, $q \in \\Red_{\tau \rta \sigma}$.
%Assumimos que $t$ é tipo $\tau^*$. Como $t$ é strongly normalizable, $q$ é strongly normalizable. Qualquer termo que se obtém por redução de $q$ obtém-se por redução de $t$, logo os elementos superficiais de qualquer tal termo são reducible. Isto é, $q \in \\Red_{\tau^*}$.
%\end{proof}

%\begin{lemma} 
%	Let $t$ be a term of type $N$. If every term $q$ such that $t\succ_1 q$ is reducible, then $t$ is reducible. 
%\end{lemma}

%\begin{lemma} \label{RedStar}
%	Let $t$ be a term of star type. If every surface element of $t$ is reducible and every term $q$ such that $t\succ_1 q$ is reducible, then $t$ is reducible. 
%\end{lemma}

%\begin{proof}
%Temos imediatamente que $t$ é strongly normalizable. Um termo que se obtém por redução de $t$ ou é $t$ ou obtém-se por redução de $q$ para $t \succ_1 q$. Como $q$ é reducible, os elementos superficiais de qualquer termo que se obtém por redução de $q$ são reducible.  Como os elementos superficiais de $t$ são reducible, concluímos que $t$ é reducible.
%\end{proof}

\begin{theorem}
	All terms in the star combinatory calculus for arithmetic  are reducible.
\end{theorem}

\begin{proof}
	Following the proof in \cite{FerreiraFerreira(17)}, we only need to show that the arithmetical constants are reducible.  It is clear that $0 \in \Red_N$. We have that if $t$ is a reducible term of type $N$ (i.e., strongly normalizable), then $St \in \Red_N$. Thus $S \in \Red_{N \rta N}$.
	
	Following Remark \ref{generictype}, we assume that $\sigma$ is $\sigma_1 \rta \cdots \rta \sigma_n \rta \rho$, with $\rho$ star type or $N$.
	We just need to prove that $R_{\sigma}$ is reducible.
	
	%We have that, by definition, if $t \in \Red_{\sigma \rta \tau}, q\in \Red_{\sigma}$ then $tq \in \Red_{\tau}$. Noting that we know that the variables are reducible, it is enough to show that the constants are.
	
	%Seja $f$ uma constante de tipo $N \rightarrow (N \rightarrow( \ldots \rightarrow (N \rightarrow N) \ldots ))$ que resulta de um símbolo funcional de $L$ com aridade $n$.
	%Temos que se $t_1,\ldots,t_n$ são termos reducible de tipo $N$ (isto é, strongly normalizable), então $ft_1\cdots t_n \in \\Red_N$. Logo $f$ é reducible.

	%It was shown in \cite{P} that $\Pi,\Sigma,\mathfrak{s},\cup$ and $\bigcup$ are reducible.
	
	%Pensemos agora no caso da aritmética. Então vejamos.
	%
	%We consider $R_{\sigma}$ [part of this proof done in \cite{Tr}]. 
	Let $t \in \Red_N, q \in \Red_\sigma, r \in \Red_{\sigma \rta N \rta \sigma}$ and $t_1,\ldots, t_n$ with $t_i \in \Red_{\sigma_i}$. We show that $R tqrt_1\cdots t_n \in \Red_{\rho}$. 
	Since $t$ is strongly normalizable, there exist finite normal forms of $t$, which we here denote $t^{(1)},\ldots,t^{(k)}$. Each $t^{(i)}$ is of the form $S\cdots St^*$ with $\mu_i(t)$ instances of $S$ at the beginning of the term, where $t^*$ is not of the form $St'^*$. We define $\mu(t):=\max_i \mu_i(t)$.
	The proof is by induction on $\mu(t)$.
	
	%Para $\mu(t)=0$, $R_{\sigma}t^{(i)}q'r't'_1\cdots 't_k$.

	For $R tqrt_1\cdots t_n$ with $\mu(t)=0$, we have that any $t'$ obtained by reduction of $t$ is not of the form $St'^*$, therefore we can only apply the conversion of $R_\sigma$ if $t'$ is $0$. Thus, in a sequence of one-step reductions of $R tqrt_1\cdots t_n$ we have two cases:
	\begin{itemize}
		\item We do not apply a conversion on $R$, therefore, as the terms involved are strongly normalizable, $R tqrt_1\cdots t_n$ also is;
		\item we reach the form $R 0q'r't'_1\cdots t'_n$ with $q \succeq q', r \succeq r', t_i \succeq t'_i$. We have $R 0q'r't'_1\cdots t'_n \succ_1 q't'_1\cdots t'_n \in \Red_\rho$ (using lemma \ref{l3}).
	\end{itemize}
	
	Thus $R tqrt_1\cdots t_n$ with $\mu(t)=0$ is strongly normalizable.
	
	We have $\mathrm{SM}(R t'q'r't'_1\cdots t'_n)=\emptyset$. As $q't'_1\cdots t'_n \in \Red_\rho$, the surface elements of terms obtained by reduction of $q't'_1\cdots t'_n$ are reducible. Thus the surface elements of terms obtained by reduction of $R tqrt_1\cdots t_n$ with $\mu(t)=0$ are reducible.
	
	%Atinge-se essa forma por reduções finitas, pois todos os termos envolvidos são strongly normalizable.
	We assume that $R tqr \in \Red_\sigma$ for all $t \in \Red_N$ with $\mu(t) \le m$ and for all $q \in \Red_\sigma, r \in \Red_{\sigma \rta N \rta \sigma}$.
	In a sequence of one-step reductions of $R tqrt_1\cdots t_n$ with $\mu(t)=m+1$ we have 3 cases (here, $q \succeq q', r \succeq r', t_i \succeq t'_i$):
	\begin{itemize}
		\item We do not apply a conversion on $R$, therefore, as the terms involved are strongly normalizable, $R tqrt_1\cdots t_n$ also is;
		\item we reach the form $R 0q'r't'_1\cdots t'_n$, case analogous to $\mu(t)=0$
		\item we reach the form $R (St')q'r't'_1\cdots t'_n$. We have $R (St')q'r't'_1\cdots t'_n \succ_1\\ r'(R t'q'r',t')t'_1\cdots t'_n \in \Red_\rho$ since $\mu(t') \le m$ (using lemma \ref{l3}).
	\end{itemize}
	%	(in fact, only the last case is possible, as consequence of teorem \ref{CR} to show afterwards).
	%Logo $R tqrt_1\cdots t_n$ with $\mu(t)=0$ é strongly normalizable.
	
	Thus $R tqrt_1\cdots t_n$ with $\mu(t)=m+1$ is strongly normalizable. We have $\mathrm{SM}(R t'q'r't'_1\cdots t'_n)=\emptyset$. As $r'(R t'q'r',t')t'_1\cdots t'_n \in \Red_\rho$, the surface elements of terms obtained by reduction of $r'(R t'q'r',t')t'_1\cdots t'_n$ are reducible. Thus the surface elements of terms obtained by reduction of $R tqrt_1\cdots t_n$ with $\mu(t)=m+1$ are reducible.
	Therefore $R tqrt_1\cdots t_n \in \Red_\rho$. We conclude that $R_{\sigma}$ is reducible.
\end{proof}

\begin{corollary}
	The star combinatory calculus on arithmetic has the property of strong normalization.
\end{corollary}

\begin{theorem} \label{CR}
	The star combinatory calculus on arithmetic has the Church-Rosser property, that is, each term has a unique normal form.
\end{theorem}

\begin{proof}
	Since we have the strong normalization property, by Newman's lemma \cite{TroelstraSchwichtenberg(96)}, it is enough to show that if $t_0 \succ_1 t_1$ and $t_0 \succ_1 t_2$, there exists a term $t_3$ such that $t_1 \succeq t_3$ and $t_2 \succeq t_3$ (that is, the calculus is weakly confluent).
	If the reductions from $t_0$ to $t_1$ and from $t_0$ to $t_2$ involve subterms which do not intersect, we obtain $t_3$ from $t_0$ by applying both conversions.
	
	We focus then on the case where one of the subterms to reduce is included in the other. We note that, for any of the conversions (including the new conversions for the recursor), the resulting terms depend on subterms of the original term, which are not altered. Thus, if $t_1$ is obtained from $t_0$ by the conversion, any other reduction $t_2$ is obtained from $t_0$ by reduction of one of the subterms $q$. We obtain $t_3$ taking $t_1$ and applying the same reductions for every instances of $q$ in the subterm resulting from the conversion. Then $t_3$ is obtained from $t_2$ applying the same conversion as from $t_0$ to $t_1$.
	
	We illustrate with the following two cases involving $R$:
	\begin{itemize}
		%\item If $t_0 \equiv M(\mathfrak{s}(q))$, $t_1 \equiv q$ and $t_2 \equiv M(\mathfrak{s}(q'))$ then we take $t_3 \equiv q'$.
		%\item If $t_0 \equiv M(\cup qr)$, $t_1 \equiv \max(Mq,Mr)$ and $t_2 \equiv M(\cup q'r)$ then we take $t_3 \equiv \max(Mq',Mr)$
		\item[] If $t_0 \equiv R(St)qr$, $t_1 \equiv r(Rtqr,t)$, $t_2 \equiv R(St)qr'$ then we take $t_3 \equiv r'(Rtqr',t)$
		\item[] If $t_0 \equiv R0qr$, $t_1 \equiv q$, $t_2 \equiv R0q'r$ then we take $t_3 \equiv q'$.
	\end{itemize}
\end{proof}

\begin{theorem}
	If $t$ is a closed normal term of type $N$ then $t$ is a numeral $\overline{n}$, i.e., it is of the form $S\cdots S0$.
\end{theorem}

\begin{proof}
	Closed terms are of the form $at_1\cdots t_m$ where $t_i$ are closed terms and $a$ is a constant. After the proof in \cite{FerreiraFerreira(17)}, we only need to verify if $R_\sigma t_1\cdots t_m$ can be a closed normal term of type $N$. The only possibility would be $R_\sigma tqrt_1\cdots t_n$ (where $\sigma$ is $\sigma_1 \rta \cdots \rta \sigma_n \rta N$). 
	%Then a closed normal term is of one of the forms: $0, S, St,$ $\cup,\cup t,\cup tq,\bigcup,\bigcup t,\bigcup tq, \Sigma,\Sigma t,\Sigma tq,\Pi,\Pi t,\mathfrak{s},\mathfrak{s}t,\\R_\sigma t_1\cdots t_i$ with $t,q,t_1,\ldots,t_i$ closed normal terms of adequate types (for $R_\sigma t_1\cdots t_i$, $0 \le i \le n+3$, where $\sigma$ is $\sigma_1 \rta \cdots \rta \sigma_n \rta \rho$ with $\rho$ star type or $N$).
	%Of these forms, the ones possible for a term of type $N$ are:
	%$0,St,R_\sigma tqrt_1\cdots t_n$ ($\sigma$ is $\sigma_1 \rta \cdots \rta \sigma_n \rta N$).
	
	By an inductive argument we may suppose that $t$ is a numeral. Therefore $R_\sigma tqrt_1\cdots t_n$ is not normal.
	
	Thus, the recursor does not contribute for the closed normal terms of type $N$, being those the terms $S\cdots S0$ with a finite (possible zero) number of $S$'s.  %Then the term has to be of the form $0$ (which is a numeral) or $St$. In case it is $St$, by the induction hypothesis $t$ is a numeral, thus $St$ also is.
	
	%O termo tem de ser então da forma $fr_1\cdots r_n$. Mostramos por indução na complexidade do termo. Se $f$ for uma constante, temos o resultado. Se $r_i$ forem termos de primeira ordem, $fr_1\cdots r_n$ também é. Logo temos que $fr_1\cdots r_n$ é um termo de primeira ordem.
\end{proof}

%\begin{definition}
%	A term $t$ of star type is said to be \emph{``set-like''} if it is constructed from terms of the form $\mathfrak{s}q$ and the constant $\cup$.
%\end{definition}

\begin{theorem}
	If $t$ is a closed normal term of star type $\rho^*$, then $t$ is set-like and $\mathrm{SM}(t)$ is a finite non-empty set of closed normal terms of type $\rho$.
\end{theorem}

\begin{proof}
	
	The proof, by induction on the term $t$, that $t$ is set-like, i.e., $t$ is of the form $\mathfrak{s}r$ or $\cup t_1t_2$,  can be found in \cite{Ferreira(20)}. The last assertion of the theorem follows immediately, noticing that: i) $\mathrm{SM}(\mathfrak{s}r)=\{r\}$, and being $\mathfrak{s}r$ a closed normal term of type $\rho^*$, then $r$ has to be a closed normal term of type $\rho$, ii) $\mathrm{SM}(\cup t_1t_2)=\mathrm{SM}(t_1) \cup \mathrm{SM}(t_2)$, following the result by induction hypothesis.   	
	
	%	Taking into account the previous proof, a closed normal term of star type has to be of one of the following forms:
	
	%	$\mathfrak{s}t,\cup tq,\bigcup tq,R_\sigma tqrt_1\cdots t_n$ ($\sigma$ is $\sigma_1 \rta \cdots \rta \sigma_n \rta \rho^*$).
	
	%	We show by induction on the complexity of the term.
	
	%	If it is $\mathfrak{s}t$ then it is set-like by definition and $\mathrm{SM}(\mathfrak{s}t)=\{t\}$, where $t$ is a closed normal term of type $\rho$.
	
	%	If it is $\cup tq$, by the induction hypothesis $t$ and $q$ are set-like, thus $\cup tq$ also is, and $\mathrm{SM}(\cup tq)=\mathrm{SM}(t) \cup \mathrm{SM}(q)$.
	
	%	In the case of the form $R_\sigma tqrt_1\cdots t_n$, if $t$ is normal and closed it is a numeral, therefore $R_\sigma tqrt_1\cdots t_n$ is not normal.
	
	%	In the case of the form $\bigcup tq$, since $t$ is normal and closed, by the induction hypothesis, $t$ is set-like, that is, $t$ is of the form $\cup t_1t_2$ or $\mathfrak{s}r$. In either case, $\bigcup tq$ is not normal.
\end{proof}

The star combinatory calculus within the arithmetic framework was also used in \cite{Ferreira(20b),EnesFerreira(23)} in the context of herbrandized functional interpretations for (respectively classical and semi-intuitionistic) second-order arithmetic. 

\subsection{Heyting Arithmetic $\HA$}

Although the herbrandized modified realizability was introduced in Section \ref{hmr} within the realm of logic, specifically  semi-intuitionistic logic, it can also be applied within the Heyting arithmetic context.

With that in view, consider the star combinatory calculus in the language of arithmetic described in subsection \ref{star_arit}. Consider also the universal axioms for the successor $S$ and the recursor $R$ (as in \cite{TroelstravanDalen(88)} or \cite{Pinto(2019)})
%\begin{itemize}
%	\item[] $R0yz = y$ 
%	\item[] $R(Sx)yz = z(Rxyz,x)$ 
	%\item $M(\cup xy) = \max(M(x),M(y))$
	%\item $M(\mathfrak{s}x) = x$
%\end{itemize}
and the (unrestricted) induction axiom scheme:
\begin{itemize}
	\item[] $A(0) \land \forall n (A(n)\to A(Sn)) \to \forall n\,A(n).$
\end{itemize}

We denote by $\HA$ the theory consisting of $\IL$ (in the language of arithmetic) with the above arithmetical axioms.

The soundness theorem has an arithmetical extension.

\begin{theorem}[Soundness, arithmetical extension]
	Let $A$ be a formula with free variables $\ua$. Let $T_{\not\exists}$ be a set of $\exists$-free formulas. If
	\begin{equation*}
		\HA + \AC + \mathrm{IP}^*_{\not\exists} + T_{\not\exists} \vdash A(\ua)
	\end{equation*}
	
	then there exist closed terms $\ut$ such that
	\begin{equation*}
		\HA +T_{\not\exists} \vdash  A_{HR}(\ua,\ut\ua)
	\end{equation*}
\end{theorem}

\begin{proof}
	The axioms for $S$ and $R$ are $\exists$-free formulas and do not require terms.
	
	Consider the axiom scheme $A(0) \land \forall n (A(n)\to A(Sn)) \to \forall n\,A(n)$.
	
	We need terms $\ur_{\uX'}$ such that:
	\begin{equation*}
		\forall \ux,\uX (A_{HR}(0,\ux) \land \forall n,\ux' (A_{HR}(n,\ux') \rta A_{HR}(Sn,\uX n\ux')) \rta \forall n\, A_{HR}(n,\ur_{\uX'}\ux\uX n))
	\end{equation*}
	
	Take $\ur_{\uX'} := \lambda \ux,\uX,n. Rn\ux\tilde{\uX}$ where $\tilde{\uX} := \lambda \ux,n.\uX n\ux$. Take $\ux,\uX$ such that
	\begin{equation*}
		A_{HR}(0,\ux) \land \forall n,\ux' (A_{HR}(n,\ux') \rta A_{HR}(Sn,\uX n\ux'))
	\end{equation*}
	
	We show that $\forall n\, A_{HR}(n,\ur_{\uX'}\ux\uX n)$ using the induction axiom. First note that $\ur_{\uX'}\ux\uX 0 = \ux$. Now assume $A_{HR}(n,\ur_{\uX'}\ux\uX n)$. By the assumption we have $A_{HR}(Sn,\uX n(\ur_{\uX'}\ux\uX n))$. Furthermore, we have:
	\begin{equation*}
		\ur_{\uX'}\ux\uX(Sn) = R(Sn)\ux\tilde{\uX} = \tilde{\uX}(Rn\ux\tilde{\uX},n)=\tilde{\uX}(\ur_{\uX'}\ux\uX n,n)=\uX n(\ur_{\uX'}\ux\uX n)
	\end{equation*}
	which gives us the result.
\end{proof}

\bmhead{Acknowledgments}

The authors are grateful to Fernando Ferreira for interesting discussions on the topic. Both authors acknowledge the support of Funda\c{c}\~{a}o para a Ci\^{e}ncia e a Tecnologia under the projects: UIDB/04561/2020, UIDB/00408/2020 and UIDP/00408/2020 and are also grateful to Centro de Matem\'{a}tica, Aplica\c{c}\~{o}es Fundamentais e Investiga\c{c}\~{a}o Operacional (Universidade de Lisboa). The first author is also grateful to  LASIGE - Computer Science and Engineering Research Centre (Universidade de Lisboa). The second author also benefitted from Funda\c{c}\~{a}o para a Ci\^{e}ncia e a Tecnologia doctoral grant 2022.12585.BD.
%\section*{Declarations}

%Some journals require declarations to be submitted in a standardised format. Please check the Instructions for Authors of the journal to which you are submitting to see if you need to complete this section. If yes, your manuscript must contain the following sections under the heading `Declarations':

%\begin{itemize}
%\item Funding
%\item Conflict of interest/Competing interests (check journal-specific guidelines for which heading to use)
%\item Ethics approval 
%\item Consent to participate
%\item Consent for publication
%\item Availability of data and materials
%\item Code availability 
%\item Authors' contributions
%\end{itemize}

%\noindent
%If any of the sections are not relevant to your manuscript, please include the heading and write `Not applicable' for that section. 

%%===================================================%%
%% For presentation purpose, we have included        %%
%% \bigskip command. please ignore this.             %%
%%===================================================%%
%\bigskip
%\begin{flushleft}%
%Editorial Policies for:

%\bigskip\noindent
%Springer journals and proceedings: \url{https://www.springer.com/gp/editorial-policies}

%\bigskip\noindent
%Nature Portfolio journals: \url{https://www.nature.com/nature-research/editorial-policies}

%\bigskip\noindent
%\textit{Scientific Reports}: \url{https://www.nature.com/srep/journal-policies/editorial-policies}

%\bigskip\noindent
%BMC journals: \url{https://www.biomedcentral.com/getpublished/editorial-policies}
%\end{flushleft}

%%===========================================================================================%%

\bibliography{sn-bibliography}% common bib file
%% if required, the content of .bbl file can be included here once bbl is generated
%%\input sn-article.bbl

\end{document}